\documentclass[10pt,reqno]{amsart}
\usepackage[T1]{fontenc}
\usepackage[utf8]{inputenc}
\usepackage{lmodern}
\usepackage[a4paper,top=0.82in,bottom=0.84in,left=0.92in,right=0.92in]{geometry}
\usepackage{amsmath,amssymb,amsfonts,amsthm,mathtools}
\usepackage{enumitem,microtype,xcolor}
\usepackage[numbers,sort&compress]{natbib}
\usepackage{hyperref}
\allowdisplaybreaks[2]
\hypersetup{
 bookmarksdepth=2,
 colorlinks=true,
 linkcolor=blue!45!black,
 citecolor=blue!45!black,
 urlcolor=blue!45!black,
 pdftitle={Sharp Weak-Trace Regularity of Stochastic Cross Integrals},
 pdfauthor={Guangqian Zhao},
 pdfsubject={Stochastic cross integrals in Hilbert spaces and sharp weak-trace regularity},
 pdfkeywords={stochastic cross integrals, Levy area, Hilbert-space semimartingales, weak trace ideals, singular values}
}

\theoremstyle{plain}
\newtheorem{theorem}{Theorem}
\newcommand{\tr}{\operatorname{tr}}
\newcommand{\wk}[1]{\lVert#1\rVert_{1,\infty}}
\newcommand{\HS}{\mathrm{HS}}

\begin{document}

\title[Sharp Weak-Trace Regularity]{Sharp Weak-Trace Regularity of Stochastic Cross Integrals}
\author{Guangqian Zhao}
\address{School of Mathematical Sciences, University of Science and Technology of China, Hefei, Anhui 230026, China}
\email{zhaoguangqian@mail.ustc.edu.cn}
\date{}

\begin{abstract}
We prove a sharp spatial regularity theorem for the It\^o and
Stratonovich cross integrals of continuous semimartingales in separable
Hilbert spaces. Without global moment assumptions, both integrals admit
versions that are almost surely jointly continuous in the two time
variables, with values in the little weak-trace ideal. This regularity
is optimal even for Wiener processes with deterministic trace-one
covariance. Given any positive sequence tending to zero, we construct
such a process for which the singular values of its L\'evy area,
multiplied by their indices, tend to zero but have termwise ratios to
the prescribed sequence tending to infinity almost surely. The same
property holds simultaneously for both self-integrals.
The upper bound follows from a rank-one cancellation in It\^o's formula
for a determinant, together with finite-rank approximation. For the
lower bound, conditional Gaussian estimates and orthogonal time
functions yield finite-dimensional spectral bounds, which are combined
through a trace-one covariance to obtain lower bounds for all
sufficiently large indices.
\end{abstract}

\maketitle

\section{Introduction}

Iterated stochastic integrals of order two provide the second-level data
of rough paths and enter differential equations driven by irregular signals
\cite{Lyons1998,LyonsQian2002,FrizVictoir2010,FrizHairer2020}.
In infinite dimensions, their interpretation depends on the tensor norm.
Ledoux--Lyons--Qian \cite{LedouxLyonsQian2002} construct canonical
L\'evy areas for Banach-space Wiener processes, including in infinite
dimensions. Their construction applies to the injective tensor product
for every Gaussian measure and to the projective tensor product under
additional conditions on that measure; in Hilbert spaces, the latter
corresponds to trace class.

For Hilbert-space Wiener processes with arbitrary trace-class covariance
$Q$, their exactness criterion
\cite[Definition~1 and Theorem~3]{LedouxLyonsQian2002} already yields
canonical rough-path lifts with second level in every Schatten class
$\mathcal S_p$, $1<p<\infty$.
Indeed, for mutually independent centered Hilbert-space Gaussian variables
$G_k,G'_k$ with covariance $Q$,
\[
 \mathbb E\Bigl\|\sum_{k=1}^N G_k\otimes G'_k\Bigr\|_{\mathcal S_p}
 \le N^{\max\{1/p,1/2\}}\tr Q,\qquad 1<p<\infty,
\]
where $\|\cdot\|_{\mathcal S_p}$ denotes the Schatten $p$-norm.
For $1<p<2$, this follows by interpolation between the trace-norm bound
$N\tr Q$ and the Hilbert--Schmidt bound $N^{1/2}\tr Q$; for $p\ge2$,
it follows from the Hilbert--Schmidt bound.
The exponent is strictly less than one, which verifies exactness.

Dereich \cite{Dereich2010} develops
finite-dimensional approximation under a canonical-lift hypothesis and
gives further sufficient conditions for that hypothesis.
For Hilbert spaces, Grong--Nilssen--Schmeding
\cite{GrongNilssenSchmeding2022} establish geometric approximation of
weakly geometric rough paths already given in a chosen Schatten tensor norm.

The singular values also connect L\'evy area to finite-dimensional
Gaussian analysis. Classical distributional formulas and double
Wiener-integral representations appear in
\cite{HelmesSchwane1983,Kree1986,FerreiroUtzet2011}.
Dela Cruz--Oberhauser \cite{DelaCruzOberhauser2026} determine
singular-value distributions and high-dimensional asymptotics for
Brownian-area matrices. Their fixed-index asymptotics
\cite[Theorem~5.5]{DelaCruzOberhauser2026} imply the finite spectral-window
lower bounds used below; we give a direct proof using conditional
covariance, orthogonal time functions, and Gaussian Gram matrices
\cite{Vershynin2012}.

We study the spatial regularity of stochastic tensor integrals.
For arbitrary continuous Hilbert-space semimartingales, both the It\^o
and Stratonovich cross integrals are jointly continuous in the little
weak-trace ideal. In the Wiener case, this sharpens the preceding
$\mathcal S_p$, $p>1$, spatial regularity to the critical weak-trace
endpoint. It also identifies a topology for finite-rank spatial
approximation. The endpoint
is sharp: the singular values multiplied by their indices tend to zero, but this convergence can be arbitrarily
slow, already for a Wiener process with deterministic trace-one
covariance. A single covariance yields the required lower bounds for
all sufficiently large indices. In particular, trace-class regularity
cannot hold in general.

The result gives a spatial regularity threshold for second-level
operators, relevant to the choice of tensor topologies in
infinite-dimensional stochastic and rough integration.
Operator-valued estimates also occur in the paracontrolled analysis
of stochastic wave equations \cite{GubinelliKochOh2024} and in random
averaging and random tensor methods for nonlinear Schr\"odinger equations
\cite{DengNahmodYue2024,DengNahmodYue2022}.
Applying this endpoint in such settings would require representing the
relevant second-order operators as stochastic cross integrals and
establishing the corresponding mapping estimates.

Let $H$ and $E$ be real separable Hilbert spaces. Denote by
$\mathcal L(H,E)$ and $\mathcal K(H,E)$ the spaces of bounded and compact
linear operators from $H$ to $E$, respectively. For a bounded operator
$K$, its adjoint is denoted by $K^*$. For $K\in\mathcal K(H,E)$,
write $s_1(K)\ge s_2(K)\ge\cdots\ge0$ for its singular values, counted
with multiplicity and completed by zeros.
We denote by $\mathcal S_1(H,E)$ the trace class, with norm
$\|K\|_1=\sum_{j\ge1}s_j(K)$, and by $\HS(H,E)$ the
Hilbert--Schmidt class, with norm
$\|K\|_{\HS}=(\sum_{j\ge1}s_j(K)^2)^{1/2}$.
Unsubscripted norms denote Hilbert
norms or operator norms, according to the argument. Define
\[
 \begin{aligned}
 \wk K&=\sup_{j\ge1}j s_j(K)\in[0,\infty],\\
 \mathcal S_{1,\infty}(H,E)
   &=\{K\in\mathcal K(H,E):\wk K<\infty\},\\
 \mathcal S^0_{1,\infty}(H,E)
   &=\{K\in\mathcal S_{1,\infty}(H,E):
                     \lim_{j\to\infty}j s_j(K)=0\}.
 \end{aligned}
\]
The little weak-trace ideal $\mathcal S^0_{1,\infty}$ is the closure
of the finite-rank operators in $\wk{\cdot}$ and contains
$\mathcal S_1$ continuously \cite[\S2.1]{Ponge2023}.
Continuity in this ideal is always understood in this quasi-norm.
For all these operator classes, a single space indicates identical
domain and codomain; for example, $\mathcal L(H)=\mathcal L(H,H)$ and
$\HS(H)=\HS(H,H)$. We omit the spaces when they are clear.
Under the Hilbert tensor--operator identification
\cite[\S\S2.1 and 5.1]{GrongNilssenSchmeding2022},
\[
 E\widehat\otimes_2H\simeq\HS(H,E),\qquad
 (y\otimes x)h=y\langle x,h\rangle_H,
 \qquad \|y\otimes x\|_1=\|y\|_E\|x\|_H.
\]

Fix $T>0$. We work on filtered probability spaces
$(\Omega,\mathcal F,(\mathcal F_t)_{0\le t\le T},\mathbb P)$
satisfying the usual conditions. A continuous Hilbert-space
semimartingale on $[0,T]$ is an adapted process $Z=M+D$, where $M$ is a
continuous local martingale and $D$ is continuous, adapted and of finite
variation in the Hilbert norm. No zero initial value or global
square-integrability is assumed. Put
\[
 Z_{s,t}=Z_t-Z_s,\qquad
 \Delta_T=\{(s,t)\in[0,T]^2:s\le t\}.
\]
For continuous semimartingales $X$ in $H$ and $Y$ in $E$ on the same
filtered probability space, their operator covariation
$[Y,X]_t\in\mathcal S_1(H,E)$ is characterized by the scalar quadratic
covariations
\[
 \langle[Y,X]_t h,e\rangle_E
   =[\langle Y,e\rangle_E,\langle X,h\rangle_H]_t,
 \qquad h\in H,\quad e\in E.
\]
Write $[X]=[X,X]$. All brackets start at zero; $[X]$ is positive and
increasing, with trace-class Stieltjes differentials. No time density
is assumed. For a continuous local martingale $M$,
$\tr[M]_T<\infty$ a.s., without an assumption on
$\mathbb E\tr[M]_T$.
The It\^o and Stratonovich cross integrals are the
$\HS(H,E)$-valued random operators
\begin{equation}\label{eq:integrals}
 \mathbb I^{Y,X,\mathrm I}_{s,t}=\int_s^tY_{s,r}\otimes dX_r,
 \qquad
 \mathbb I^{Y,X}_{s,t}=\mathbb I^{Y,X,\mathrm I}_{s,t}
                 +\tfrac12([Y,X]_t-[Y,X]_s).
\end{equation}
The stochastic integral is defined by localization
\cite{Metivier1982}. The martingale integrals have continuous
Hilbert--Schmidt paths; integrals against finite-variation parts and
the bracket correction have continuous trace-class paths. For a
continuous $H$-valued semimartingale $Z$, its L\'evy area is
\[
 A^Z_{s,t}:=\tfrac12\bigl(\mathbb I^{Z,Z}_{s,t}
                    -(\mathbb I^{Z,Z}_{s,t})^*\bigr)\in\HS(H).
\]

\section{Main result and proof}

Our main result is the following.

\begin{theorem}\label{thm:main}
\leavevmode
\begin{enumerate}[label=\textup{(\roman*)},leftmargin=*,labelsep=.5em]
\item Let $X$ and $Y$ be continuous semimartingales in $H$ and $E$
with respect to the same filtration.
Their It\^o and Stratonovich cross integrals admit versions satisfying
\[
 \mathbb I^{Y,X,\mathrm I},\ \mathbb I^{Y,X}
 \in C\bigl(\Delta_T;\mathcal S^0_{1,\infty}(H,E)\bigr)
 \qquad\textup{a.s.}
\]
In particular, for every continuous $H$-valued semimartingale $X$,
\[
 A^X\in C\bigl(\Delta_T;\mathcal S^0_{1,\infty}(H)\bigr)
 \qquad\textup{a.s.}
\]

\item Suppose that $H$ and $E$ are infinite-dimensional.
For every $T>0$ and
every positive sequence $\rho_j\to0$, there is a filtered probability
space carrying a jointly Gaussian pair of Wiener processes $(X,Y)$ in
$H\times E$ on $[0,T]$ relative to this filtration such that,
simultaneously for
\[
 L\in\{\mathbb I^{Y,X}_{0,T},\mathbb I^{Y,X,\mathrm I}_{0,T}\}
     \subset\HS(H,E),
\]
we have
\begin{equation}\label{eq:sharp}
 L\in\mathcal S^0_{1,\infty}(H,E),\qquad
 \lim_{j\to\infty}\frac{j s_j(L)}{\rho_j}=\infty
 \qquad\textup{a.s.}
\end{equation}
When $E=H$, one may take $X=Y=M$ with $[M]_t=tQ$, where
$Q\in\mathcal S_1(H)$ is deterministic, nonnegative and self-adjoint,
with $\tr Q=1$. Then \eqref{eq:sharp} also holds for $L=A^M_{0,T}$
on the same event.
\end{enumerate}
\end{theorem}

For $1\le p<\infty$, the Schatten--Lorentz ideal is
\[
 \mathcal S_{1,p}(H,E)
 :=\left\{K\in\mathcal K(H,E):
       \sum_{j\ge1}\frac{(j s_j(K))^p}{j}<\infty\right\}.
\]
Taking $\rho_j=1/\log\log(e^e+j)$ in part~(ii) gives a single
trace-one covariance $Q$ for which the terminal area and both
the It\^o and Stratonovich self-integrals of a $Q$-Wiener process
lie outside every $\mathcal S_{1,p}(H)$, $1\le p<\infty$, on a common
event of probability one. Indeed, each of these operators $L$ satisfies
$j s_j(L)\ge\rho_j$ for all sufficiently large $j$ on this event,
whereas $\sum_{j\ge1}\rho_j^p/j=\infty$ for every finite $p$.

We first record the singular-value inequalities used below
\cite{Simon2005}. For $K,L\in\mathcal K(H,E)$ and bounded operators
$B,C$ for which the compositions are defined,
\begin{equation}\label{eq:singular-values}
 \begin{aligned}
 s_j(BKC)&\le\|B\|\,s_j(K)\,\|C\|,\\
 s_{j+k-1}(K+L)&\le s_j(K)+s_k(L),\\
 |s_j(K)-s_j(L)|&\le\|K-L\|,
 \end{aligned}
 \qquad j,k\ge1.
\end{equation}
We also use $s_j(K^*)=s_j(K)$ and
$s_j(K^*K)=s_j(KK^*)=s_j(K)^2$.
For self-adjoint operators on the same Hilbert space, $S\le R$ means
$\langle(R-S)h,h\rangle\ge0$ for every $h$.
The min--max principle gives, for nonnegative compact operators,
\begin{equation}\label{eq:positive-order}
 0\le S\le R\quad\Longrightarrow\quad s_j(S)\le s_j(R),
 \qquad j\ge1.
\end{equation}
In particular, compression by orthogonal projections cannot increase
singular values; left or right multiplication by surjective linear
isometries preserves them.
The sum inequality in \eqref{eq:singular-values} implies
\begin{equation}\label{eq:quasi-triangle}
 \wk{K+L}\le2\bigl(\wk K+\wk L\bigr).
\end{equation}
For finite-rank $K$, we further have
\begin{equation}\label{eq:finite-rank-norms}
 \wk K\le\|K\|_1
       \le\sqrt{\operatorname{rank}K}\,\|K\|_{\HS}.
\end{equation}

To prove the theorem, put $Z=(X,Y)=M+D$ in $H\oplus E$
and $G_t=\mathbb I^{M,M}_{0,t}$. Tensor integration by parts and additivity give
\begin{equation}\label{eq:reduction}
 \begin{aligned}
 \mathbb I^{Z,Z}_{0,t}-G_t
   &=D_{0,t}\otimes M_{0,t}
     +\int_0^t(Z_{0,r}\otimes dD_r-dD_r\otimes M_{0,r}),\\
 \mathbb I^{Z,Z}_{s,t}
   &=\mathbb I^{Z,Z}_{0,t}-\mathbb I^{Z,Z}_{0,s}-Z_{0,s}\otimes Z_{s,t}.
 \end{aligned}
\end{equation}
The two identities show that, apart from $G_t-G_s$, all remaining terms
in the interval self-integral are trace-norm continuous; so is the It\^o
correction. If $p_H,p_E$ are the coordinate projections from $H\oplus E$,
then
\[
 \mathbb I^{Y,X}_{s,t}=p_E\mathbb I^{Z,Z}_{s,t}p_H^*.
\]
The same identity holds for the It\^o integrals.
By \eqref{eq:singular-values}, these block compressions cannot increase
singular values. Thus part~(i) reduces to
$G\in C([0,T];\mathcal S^0_{1,\infty})$; the second identity supplies the
jointly continuous interval version. For either terminal self-integral
$L$ of a Wiener process $M$, its area $A=A^M_{0,T}$ satisfies, by
\eqref{eq:singular-values},
\[
 A=(L-L^*)/2,\qquad s_{2j}(A)\le s_j(L),\qquad s_{2j}(A)\le s_j(A).
\]
Hence part~(ii) needs only an area lower bound for $j s_{2j}(A)/\rho_j$.

\begin{proof}[Proof of Theorem~\ref{thm:main}(i)]
Choose deterministic finite-rank orthogonal projections $P_n\uparrow I$
strongly on $H\oplus E$, and write $P_n^\perp=I-P_n$. The decomposition
\begin{equation}\label{eq:main-part}
 G_t=\underbrace{P_nG_t+P_n^\perp G_tP_n}_{\text{finite-rank part}}
                +\underbrace{P_n^\perp G_tP_n^\perp}_{\text{remainder}}
\end{equation}
has a finite-rank part of rank at most $2\operatorname{rank}P_n$.
Its Hilbert--Schmidt continuity implies trace-norm continuity by
\eqref{eq:finite-rank-norms}.
The remainder is the self-integral of $P_n^\perp M$, whose bracket
trace tends to zero.
By finite-rank closure and \eqref{eq:reduction}, it suffices to prove
\begin{equation}\label{eq:tail}
 \varepsilon_n:=\sup_{t\le T}\wk{P_n^\perp G_tP_n^\perp}
             \downarrow0\qquad\textup{a.s.}
\end{equation}
The suprema may initially be infinite; they are measurable since
\eqref{eq:singular-values} makes each singular value continuous in time.

For a dimension-free maximal bound, let $U$ be a continuous local
martingale in $\mathbb R^d$, and put
\[
 u_t=U_t-U_0,\qquad
 K_t=\mathbb I^{U,U}_{0,t}
     =\int_0^tu_r\otimes dU_r+\tfrac12[U]_t,\qquad
 V_t=\tr[U]_t.
\]
Here $K_t\in\mathcal L(\mathbb R^d)$ and $[U]$ is the matrix quadratic
variation. By the definition of $K$ and the It\^o product rule,
\[
 \begin{aligned}
 dK_t&=u_t\otimes dU_t+\tfrac12d[U]_t,\\
 d(u_t\otimes u_t)
   &=u_t\otimes dU_t+dU_t\otimes u_t+d[U]_t
     =dK_t+dK_t^*.
 \end{aligned}
\]
Since $u_0=0$ and $K_0=0$, integration yields
\[
 K_t+K_t^*=u_t\otimes u_t.
\]
For $\theta>0$, set $R_t=I_d+2\theta K_t$ and suppress time subscripts
in the following operator identities. Then
\begin{equation}\label{eq:positive}
 R^*R=I+2\theta u\otimes u+4\theta^2K^*K
             \ge I+4\theta^2K^*K.
\end{equation}
Thus $\|R^{-1}\|\le1$ and $\det R>0$, by continuity from $\theta=0$.
By \eqref{eq:positive-order},
\begin{equation}\label{eq:spectral-readout}
 \det R=\prod_{\ell=1}^d s_\ell(R)
 \ge\prod_{\ell=1}^j\sqrt{1+4\theta^2s_\ell(K)^2}
 \ge(2\theta s_j(K))^j,\qquad 1\le j\le d.
\end{equation}
The same determinant is affine along the martingale directions:
\begin{equation}\label{eq:hessian}
 v\longmapsto\det(R+u\otimes v)\ \text{is affine},\qquad
 D^2\det(R)[u\otimes h,u\otimes k]=0.
\end{equation}
Indeed, every term containing two perturbation columns has parallel
columns. Since the martingale part of $dR$ is $2\theta u\otimes dU$,
the Hessian contribution in It\^o's formula
\cite[Chapter~IV]{RevuzYor1999} vanishes. Jacobi's formula gives
\begin{equation}\label{eq:det-differential}
 d\det R=2\theta\det R\langle R^{-1}u,dU\rangle
                   +\theta\det R\tr(R^{-1}d[U]).
\end{equation}
Trace duality and positivity of $d[U]$ give, as an inequality of measures,
\[
 |\tr(R^{-1}d[U])|\le\|R^{-1}\|\tr(d[U])\le dV.
\]
The finite-variation part of $d(e^{-\theta V}\det R)$ is therefore
\[
 \theta e^{-\theta V}\det R\bigl(\tr(R^{-1}d[U])-dV\bigr)\le0.
\]
Thus $e^{-\theta V}\det R$ is a nonnegative local supermartingale,
hence a supermartingale, starting at one. For $a,v>0$ and $1\le j\le d$,
its maximal inequality and
\eqref{eq:spectral-readout} give
\[
 \begin{aligned}
 \{\sup_{t\le T}j s_j(K_t)>a,\ V_T\le v\}
 &\subset\left\{\sup_{t\le T}e^{-\theta V_t}\det R_t
 >e^{-\theta v}\left(\frac{2\theta a}{j}\right)^j\right\},\\
 \mathbb P\{\sup_{t\le T}j s_j(K_t)>a,\ V_T\le v\}
 &\le e^{\theta v}\left(\frac{j}{2\theta a}\right)^j.
 \end{aligned}
\]
Taking $\theta=j/v$ and applying the union bound over $j$, we obtain,
for $a\ge ev$,
\[
 \mathbb P\{\sup_{t\le T}\wk{K_t}>a,\ V_T\le v\}
 \le\sum_{j\ge1}\left(\frac{ev}{2a}\right)^j\le\frac{ev}{a}.
\]
For $a<ev$ the bound $1$ applies. Consequently,
\begin{equation}\label{eq:maximal}
 \mathbb P\{\sup_{t\le T}\wk{K_t}>a,\ V_T\le v\}
                 \le\min\{1,ev/a\},\qquad a,v>0.
\end{equation}

For $\ell>n$, the finite-dimensional martingale $(P_\ell-P_n)M$ has
maximal self-integral quasi-norm and terminal bracket trace
\[
 \begin{aligned}
 \varepsilon_{n,\ell}
   &:=\sup_{t\le T}\wk{(P_\ell-P_n)G_t(P_\ell-P_n)},\\
 v_{n,\ell}
   &:=\tr\bigl((P_\ell-P_n)[M]_T(P_\ell-P_n)\bigr).
 \end{aligned}
\]
Applying \eqref{eq:maximal} and splitting according to
$\{v_{n,\ell}>v\}$ gives
\[
 \mathbb P\{\varepsilon_{n,\ell}>a\}
 \le\mathbb P\{v_{n,\ell}>v\}+ev/a.
\]
For fixed $n$, nested compression and Hilbert--Schmidt convergence,
together with \eqref{eq:singular-values}, give, as $\ell\to\infty$,
\[
 \varepsilon_{n,\ell}\uparrow\varepsilon_n,\qquad
 v_{n,\ell}\uparrow v_n:=\|P_n^\perp[M]_T^{1/2}\|_{\HS}^2.
\]
Monotone convergence of the strict threshold events therefore yields
\[
 \mathbb P\{\varepsilon_n>a\}\le\mathbb P\{v_n>v\}+ev/a.
\]
Since $[M]_T^{1/2}$ is Hilbert--Schmidt and $P_n^\perp\downarrow0$
strongly, $v_n\downarrow0$ a.s.
Let $n\to\infty$ and then $v\downarrow0$. Thus $\varepsilon_n\to0$
in probability; since nested compression makes $\varepsilon_n$
nonincreasing, its almost sure limit is zero. The continuous finite-rank
parts in \eqref{eq:main-part} therefore converge uniformly to $G$
in $\wk{\cdot}$. By finite-rank closure and \eqref{eq:quasi-triangle},
$G\in C([0,T];\mathcal S^0_{1,\infty})$, completing part~(i).
\end{proof}

\begin{proof}[Proof of Theorem~\ref{thm:main}(ii)]
For an integer $m\ge1$ and $0<\delta<1$, we first construct a
finite-dimensional Wiener process on $[0,1]$ with trace-one covariance whose
terminal area $B$ satisfies
\begin{equation}\label{eq:window-target}
 \mathbb P\left\{\min_{1\le j\le m}j s_{2j}(B)
                         \ge\frac1{8\pi}\right\}\ge1-\delta.
\end{equation}

For an integer $n\ge2m$ to be chosen below, let $\beta=(X,Y)$, where
$X,Y$ are independent standard $n$-dimensional Brownian motions.
By integration by parts, the upper-right $n\times n$ block
of $A^\beta_{0,1}$ is
\begin{equation}\label{eq:corner}
 C=\int_0^1X_t\otimes dY_t-\tfrac12X_1\otimes Y_1
 \in\mathbb R^{n\times n}.
\end{equation}
Conditionally on the entire path $X$, the columns of $C$ are independent
$N(0,\Sigma)$ vectors, where
\begin{equation}\label{eq:conditional-covariance}
 \Sigma=\int_0^1z_t\otimes z_t\,dt,\qquad z_t=X_t-\tfrac12X_1.
\end{equation}
Define the random bounded linear operator
$\mathcal A:L^2(0,1)\to\mathbb R^n$ by
$\mathcal A f=\int_0^1f(t)z_t\,dt$; then $\Sigma=\mathcal A\mathcal A^*$.
For deterministic $f\in L^2(0,1)$, stochastic Fubini gives
\begin{equation}\label{eq:time-operator}
 \mathcal A f=\int_0^1(\mathcal T f)(s)\,dX_s,\qquad
 (\mathcal T f)(s)=\int_s^1f(t)\,dt-\tfrac12\int_0^1f(t)\,dt.
\end{equation}
The operator $\mathcal T$ maps $L^2(0,1)$ into itself. The orthonormal
family $(f_1,c_1,\ldots,f_m,c_m)$ in $L^2(0,1)$ satisfies
\begin{equation}\label{eq:time-pairs}
 \begin{gathered}
 f_j(t)=\sqrt2\sin((2j-1)\pi t),\qquad
 c_j(t)=\sqrt2\cos((2j-1)\pi t),\\
 \mathcal T f_j=\lambda_jc_j,\qquad
 \mathcal T c_j=-\lambda_jf_j,\qquad
 \lambda_j=\frac1{(2j-1)\pi},\quad 1\le j\le m.
 \end{gathered}
\end{equation}
By \eqref{eq:time-operator} and orthonormality,
\[
 F=\bigl(\lambda_1^{-1}\mathcal A f_1,\lambda_1^{-1}\mathcal A c_1,
         \ldots,\lambda_m^{-1}\mathcal A f_m,\lambda_m^{-1}\mathcal A c_m\bigr)
\]
is a standard $n\times2m$ Gaussian matrix, meaning that its entries
are independent $N(0,1)$ random variables. With
$\Lambda_m=\operatorname{diag}(\lambda_1,\lambda_1,\ldots,\lambda_m,\lambda_m)$,
Bessel's inequality gives
\begin{equation}\label{eq:bessel}
 \Sigma=\mathcal A\mathcal A^*\ge F\Lambda_m^2F^*.
\end{equation}

For fixed $m$, the law of large numbers gives
$n^{-1}F^*F\to I_{2m}$ in probability as $n\to\infty$.
On $\{F^*F\ge(n/2)I_{2m}\}$, the matrix
$\Lambda_mF^*F\Lambda_m$ dominates $(n/2)\Lambda_m^2$ and has the
same nonzero eigenvalues as $F\Lambda_m^2F^*$.
Thus \eqref{eq:bessel} and \eqref{eq:positive-order} yield, for all
sufficiently large $n$,
\begin{equation}\label{eq:covariance-lower}
 \mathbb P\left\{
 s_{2j}(\Sigma^{1/2})\ge\sqrt{n/2}\,\lambda_j,
                   \ \text{for all }1\le j\le m\right\}\ge1-\delta/2.
\end{equation}

For each fixed conditioning path $X$, express the columns of $C$ in an
orthonormal eigenbasis of $\Sigma$, ordered by nonincreasing eigenvalues.
The first $2m$ rows then have conditional distribution $D_\Sigma W$, where
\[
 D_\Sigma=\operatorname{diag}\bigl(s_i(\Sigma^{1/2})\bigr)_{i=1}^{2m},
 \qquad W\in\mathbb R^{2m\times n}\text{ is standard Gaussian}.
\]
The same Gram-matrix convergence for $n^{-1}WW^*$, row compression, and
\eqref{eq:positive-order} give, for all sufficiently large $n$,
\begin{equation}\label{eq:gaussian-area}
 \mathbb P\left\{
 s_i(C)\ge\sqrt{n/2}\,s_i(\Sigma^{1/2}),\quad
 \text{for all }1\le i\le2m\,\middle|\,X\right\}
 \ge1-\delta/2\qquad\textup{a.s.}
\end{equation}
The choice of $n$ depends only on $m$ and $\delta$. Fix $n$ large enough
for both estimates. Taking expectations in \eqref{eq:gaussian-area}
and applying the union bound with \eqref{eq:covariance-lower}, we obtain
\[
 \mathbb P\left\{s_{2j}(C)\ge\frac n2\lambda_j,
              \ \text{for all }1\le j\le m\right\}\ge1-\delta.
\]

For $B_{2n}:=A^{\beta/\sqrt{2n}}_{0,1}=(2n)^{-1}A^\beta_{0,1}$,
the compression inequality in \eqref{eq:singular-values} gives,
on the same event,
\begin{equation}\label{eq:normalized-window}
 j s_{2j}(B_{2n})
 \ge\frac{j}{2n}s_{2j}(C)
 \ge\frac{j\lambda_j}{4}
 \ge\frac1{8\pi},\qquad 1\le j\le m.
\end{equation}
Since $\tr(I_{2n}/(2n))=1$, this proves \eqref{eq:window-target} with
$B=B_{2n}$.

Given $\rho_j>0$ with $\rho_j\to0$, choose strictly increasing integers
$r_k$ such that
\[
 \rho_j\le4^{-k}\qquad(j\ge r_k,\ k\ge1).
\]
For each $k$, apply \eqref{eq:window-target} with
$m=r_{k+1}$ and $\delta=2^{-k}$ to fix a deterministic even dimension
$N_k$. On the Hilbert direct sum
$\mathcal H_*:=\bigoplus_{k\ge1}\mathbb R^{N_k}$, define
\begin{equation}\label{eq:covariance}
 Q=\bigoplus_{k\ge1}\frac{2^{-k}}{N_k}I_{N_k}
       \in\mathcal S_1(\mathcal H_*),\qquad
 \tr Q=\sum_{k\ge1}2^{-k}=1.
\end{equation}
Let $M$ be the $Q$-Wiener process and $A=A^M_{0,1}$.
For the coordinate projections
$\Pi_k:\mathcal H_*\to\mathbb R^{N_k}$, compression and Brownian
scaling give
\[
 \Pi_k A\Pi_k^*\stackrel{d}{=}2^{-k}B_{N_k}.
\]
Define the block events
\[
 E_k:=\left\{\min_{1\le j\le r_{k+1}}j s_{2j}(\Pi_kA\Pi_k^*)
                         \ge\frac{2^{-k}}{8\pi}\right\}.
\]
By \eqref{eq:window-target} and the first Borel--Cantelli lemma,
\[
 \mathbb P(E_k^c)\le2^{-k},\qquad
 \Omega_0:=\bigcup_{k_0\ge1}\bigcap_{k\ge k_0}E_k,
 \qquad \mathbb P(\Omega_0)=1.
\]
On $\Omega_0$, for all sufficiently large $k$ and every $j$ with
$r_k\le j<r_{k+1}$, compression gives
\begin{equation}\label{eq:all-indices}
 \frac{j s_{2j}(A)}{\rho_j}
 \ge\frac{j s_{2j}(\Pi_k A\Pi_k^*)}{\rho_j}
 \ge\frac{2^{-k}}{8\pi\rho_j}
 \ge\frac{2^k}{8\pi}\longrightarrow\infty.
\end{equation}
These intervals cover all sufficiently large integers, and $k\to\infty$
as $j\to\infty$. The area comparison above and part~(i)
give \eqref{eq:sharp} for the area and both self-integrals on a common
event of probability one.
Finally, choose surjective linear isometries $J_H:\mathcal H_*\to H$
and $J_E:\mathcal H_*\to E$, and set $(X,Y)=(J_HM,J_EM)$.
Their cross integrals are $J_ELJ_H^*$, where $L$ is the corresponding
self-integral of $M$; these operators have the same singular values
as $L$. Take $J_H=J_E$ when $H=E$; the covariance is then
$J_HQJ_H^*$, still of trace one.
For general $T>0$, apply the same isometries to
$M^{(T)}_t=\sqrt T\,M_{t/T}$, $0\le t\le T$; its bracket is $tQ$,
and the terminal area and both cross integrals are multiplied by $T$.
\end{proof}

\section*{Declaration of AI use}
The author used ChatGPT (GPT-5.6 and GPT-6.0; OpenAI) interactively in
developing proofs and refining the presentation. The author has
independently verified all proofs and takes full responsibility for
the content of the manuscript.

\end{document}